\documentclass[12pt,a4paper]{article}
\usepackage[T1]{fontenc}
\usepackage[utf8]{inputenc}
\input{glyphtounicode}
\usepackage{lmodern}
\usepackage{amsmath,amssymb,amsthm,mathtools,mathrsfs}
\usepackage[width=15.6cm,height=23.4cm,centering]{geometry}
\usepackage{microtype}
\usepackage{booktabs,array}
\usepackage[hidelinks]{hyperref}
\hypersetup{pdftitle={Billiards, Refraction, and Blow-Up for Kirchhoff Equations},pdfsubject={Transverse heteroclinic connections and forced blow-up for Kirchhoff equations},pdfkeywords={Kirchhoff equation, billiards, refraction, heteroclinic connection, blow-up}}
\usepackage{fancyhdr}
\usepackage{enumitem}
\setlist{itemsep=3pt,topsep=5pt}
\numberwithin{equation}{section}
\theoremstyle{plain}
\newtheorem{theorem}{Theorem}[section]
\newtheorem{proposition}[theorem]{Proposition}
\newtheorem{lemma}[theorem]{Lemma}
\newtheorem{corollary}[theorem]{Corollary}
\theoremstyle{definition}

\theoremstyle{remark}
\newtheorem{remark}[theorem]{Remark}

\newcommand{\Z}{\mathbb Z}
\newcommand{\eps}{\varepsilon}
\newcommand{\dd}{\,d}
\newcommand{\Dnu}{D_\nu}
\newcommand{\Psinu}{\Psi_\nu}
\newcommand{\nn}{\mathbf n}
\newcommand{\ttt}{\mathbf t}
\newcommand{\er}{e_r}
\newcommand{\et}{e_\theta}
\newcommand{\cH}{\mathscr H}
\newcommand{\cB}{\mathcal B}
\DeclareMathOperator{\diag}{diag}

\DeclareMathOperator{\supp}{supp}

\title{Billiards, Refraction, and Blow-Up for Kirchhoff Equations}
\author{Marina Ghisi\vspace{1ex}\\ 
{\normalsize Università degli Studi di Pisa} \\
{\normalsize Dipartimento di Matematica}\\ 
{\normalsize PISA (Italy)}\\
{\normalsize e-mail: \texttt{marina.ghisi@unipi.it}}
\and
Massimo Gobbino\vspace{1ex}\\ 
{\normalsize Università degli Studi di Pisa} \\
{\normalsize Dipartimento di Matematica}\\ 
{\normalsize PISA (Italy)}\\  
{\normalsize e-mail: \texttt{massimo.gobbino@unipi.it}}
}
\date{\today}
\begin{document}
\maketitle
\thispagestyle{empty}
\begin{abstract}

We start from a simple problem in geometric optics. A particle moves between two homothetic ellipses, with refraction at the inner interface and reflection at the outer one. In a suitable nearly circular regime, the small geometric anisotropy of the ellipses is amplified by the strong refraction, and the corresponding return map develops a transverse heteroclinic connection between the two axial motions.

We turn this geometric mechanism into a construction for a two-mode Kirchhoff system while keeping the standard quadratic elastic variable throughout. The singular optical model is first reduced to an explicit kick--drift map, for which the heteroclinic connection is obtained by a contraction argument. The same argument also gives exponential convergence along the two tails and transversality. We then show that the connection persists for the genuine elliptic billiard and, subsequently, through a smooth regularization of the reflecting and refracting interfaces. Finally, a small positive background is added to the Kirchhoff coefficient, making it uniformly positive without destroying the hyperbolicity of the axial modes or the transverse connection.

This produces a smooth and uniformly positive Kirchhoff coefficient admitting a transverse heteroclinic orbit between two simple modes. Combined with the road-map theorem developed in our previous work, the construction yields a finite-time blow-up example for a forced abstract Kirchhoff equation with a regular forcing term.

\end{abstract}

\section{Introduction and main result}\label{sec:intro}

Let $m:[0,\infty)\to(0,\infty)$ and $\nu>1$.  We study
\begin{equation}\label{eq:kids}
\begin{cases}
 x''+m(x^2+\nu y^2)x=0,\\
 y''+\nu m(x^2+\nu y^2)y=0.
\end{cases}
\end{equation}
Writing
\[
 \Psinu(x,y)=x^2+\nu y^2,
 \qquad \Dnu=\diag(1,\nu),
 \qquad M(s)=\int_0^s m(\sigma)\dd\sigma,
\]
the system is Hamiltonian with
\begin{equation}\label{eq:hamiltonian}
 \cH(R,V)=\frac12|V|^2+\frac12M(\Psinu(R)),
 \qquad R=(x,y),\quad V=(x',y').
\end{equation}
The kinetic energy is Euclidean.  Both coordinate axes are invariant, and
the corresponding nonconstant periodic solutions are called simple modes.
In all statements below, hyperbolicity and transversality refer to the
Poincare dynamics on a fixed energy surface.  We do not use hyperbolicity
in the full four-dimensional phase space, where the autonomous
Hamiltonian structure supplies the usual neutral directions.

The connection sought here should not be confused with an orbit that
leaves an axis at a finite time. If both the position and the velocity of
one component vanish at a finite time, uniqueness keeps that component
identically zero. A heteroclinic orbit instead has two genuinely active
components during its transition and approaches different axial periodic
orbits only as $t\to\pm\infty$. The estimates in the main theorem measure
the decay of the component that becomes inactive at each end.

For clarity, a \emph{transverse heteroclinic connection} on an energy
surface means an orbit contained in
$W^u(\Gamma_x)\cap W^s(\Gamma_y)$ whose two invariant manifolds span the
tangent space of that energy surface along the orbit. On a transverse
section this is simply an intersection of two curves with distinct tangent
lines. It is this two-dimensional formulation that will be used below.

Instability of simple modes for suitable Kirchhoff nonlinearities was
studied, among other works, in \cite{GGunstable}.  The existence of an
unstable simple mode does not by itself imply a connection to a different
one.  The road map in \cite{GGroad} isolates precisely such a connection
as a sufficient input for constructing a forced infinite-dimensional
blow-up example.  Our task here is the finite-dimensional dynamical input.

The construction has two small-scale stages.  First, a nearly circular
geometry is combined with a much smaller speed in the annulus.  A small
anisotropic deflection then becomes a finite kick after normalization by
the annular speed.  Second, for a fixed member of this family, the sharp
interfaces are replaced by thin smooth layers.  These two limits are
\emph{ordered}, not simultaneous.

There are three distinct tasks in the proof. We first derive a limiting
collision map whose formula is explicit. We then construct an orbit of
that map by a contraction, rather than by numerical shooting. Finally,
we use transversality to replace both idealizations---the limiting
geometry and the sharp interfaces---by genuine smooth dynamics.
The same planar quadratic form $x^2+\nu y^2$ is used at every stage.

Define the smooth step function
\begin{equation}\label{eq:chi}
 B(t)=\begin{cases}e^{-1/t},&t>0,\\0,&t\le0,\end{cases}
 \qquad
 \chi(z)=\frac{B(z+1)}{B(z+1)+B(1-z)}.
\end{equation}
Then $\chi=0$ on $(-\infty,-1]$, $\chi=1$ on $[1,\infty)$, and
$\chi'>0$ on $(-1,1)$.

\begin{theorem}[Smooth, strictly positive Kirchhoff realization]\label{thm:main}
There exists $\eps_0\in(0,1)$ with the following property.  For every
$0<\eps<\eps_0$, set
\begin{equation}\label{eq:physicalfamily}
 \nu=1+\eps,\qquad a=\frac12,\qquad b=1,
 \qquad E=\frac12,\qquad h=\frac12-\frac{\eps^2}{512}.
\end{equation}
There exists $\delta_0(\eps)>0$ such that, for every
$0<\delta<\delta_0(\eps)$, there exists
$\eta_0(\eps,\delta)>0$ for which the following holds whenever
$0<\eta<\eta_0(\eps,\delta)$.
The profile
\begin{equation}\label{eq:mfinal}
 m(s)=\eta+
 \frac h\delta\chi'\!\left(\frac{s-a^2}{\delta}\right)
 +\frac{1-h}{\delta}\chi'\!\left(\frac{s-b^2}{\delta}\right)
\end{equation}
is smooth, bounded, and satisfies $\inf_{s\ge0}m(s)\ge\eta>0$.
On the energy surface $\cH=E/2=1/4$, system \eqref{eq:kids} has two
hyperbolic axial periodic orbits $\Gamma_x$ and $\Gamma_y$ and a
transverse heteroclinic orbit from $\Gamma_x$ to $\Gamma_y$.

For a parametrization $(x(t),y(t))$ of this orbit there are constants
$A_0,B_0>0$ such that
\begin{equation}\label{eq:hca-decay}
\begin{aligned}
 |y'(t)|^2+\nu|y(t)|^2&\le B_0 e^{-A_0|t|},&&t\le0,\\
 |x'(t)|^2+|x(t)|^2&\le B_0 e^{-A_0t},&&t\ge0.
\end{aligned}
\end{equation}
The same system also has a transverse connection in the reverse direction.
\end{theorem}

The primitive of \eqref{eq:mfinal}, normalized at zero, is
\begin{equation}\label{eq:Mfinal}
 M(s)=h\chi\!\left(\frac{s-a^2}{\delta}\right)
 +(1-h)\chi\!\left(\frac{s-b^2}{\delta}\right)+\eta s,
\end{equation}
provided the layers are disjoint and lie away from zero.  No angular
correction and no non-Hilbertian norm is introduced.  The function $M$ is
strictly increasing, whereas $m$ need not be monotone: it consists of two
smooth peaks on a positive background.

Theorem~\ref{thm:main} supplies the hypothesis of \cite[Definition~2.1]{GGroad} with \(\lambda=\sqrt\nu\). As a consequence, the road-map theorem yields a finite-time blow-up example for a forced abstract Kirchhoff equation, with forcing belonging to every Gevrey class of order strictly larger than one. The precise statement is given in Section~\ref{sec:road}. In this sense, the construction reaches the boundary of the classical global existence theory: quasi-analytic regularity prevents the phenomenon, whereas immediately beyond the quasi-analytic regime finite-time blow-up can already occur.

\subsection*{How the proof is organized}
The roles of the small parameters are different:
\begin{center}
\renewcommand{\arraystretch}{1.18}
\begin{tabular}{@{}p{1.8cm}p{6.45cm}p{6.35cm}@{}}
\toprule
Parameter & What is changed & What is kept fixed \\
\midrule
$\eps$ & Ellipticity and speed ratio $q=\eps/16$
 & $\alpha=1/2$ and the scaled kick strength \\
$\delta$ & Thickness of the two transition layers
 & One sufficiently small positive $\eps$ \\
$\eta$ & A positive background in $m$
 & The already regularized system, including $\delta$ \\
\bottomrule
\end{tabular}
\end{center}
Sections~\ref{sec:step}--\ref{sec:limit} explain the geometry and compute
the limit map. Section~\ref{sec:contraction} contains the contraction and
the transversality argument. Section~\ref{sec:ellipticpersistence} isolates
the persistence principle, including why only finitely many iterates need
to be compared. Sections~\ref{sec:layers}--\ref{sec:realization} replace the
interfaces by smooth layers and then make $m$ positive.
The infinite-dimensional statement in Section~\ref{sec:road} is an
application of the external road-map theorem, not another construction
hidden in the planar proof.

We use $E$ for twice the mechanical energy, $q$ for the ratio of speeds,
$p$ for a normalized tangential velocity on the section, and $n\in\Z$
for the return index. The symbol $\theta$ is the eccentric angle of an
ellipse, not its polar angle when $\eps>0$.

\paragraph{Status of this preliminary version and AI assistance.}

This is a very preliminary version of the paper, posted in order to make the construction and its main ideas available without waiting for the much slower process of a complete human verification and rewriting. A subsequent version is planned, with a line-by-line human check of all arguments, additional details, and a more extensive bibliography. This may take some time; given recent developments, by then someone may already have produced a Lean-verified proof without forcing.

The geometric ideas behind the construction long predate the present work with AI assistance. The billiard interpretation of the two-mode Kirchhoff system already appeared in our 2001 paper~\cite{GGunstable}. The possibility of introducing two potential plateaux, with refraction at the inner interface and reflection at the outer one, was considered by us much later and discussed in several talks, but we had never succeeded in completing the corresponding calculations for elliptic interfaces or in finding a satisfactory mechanism producing the desired heteroclinic connection.

The present draft was prepared with substantial assistance from OpenAI's ChatGPT, using the GPT-5.6 Sol model in September 2026. During an extended mathematical discussion, it helped recompute and check the elliptic reflection--refraction maps, helped identify the contraction argument used to construct the limiting heteroclinic orbit, and pointed us to the appropriate perturbative results in the literature needed to justify its persistence under smoothing. Remarkably, the decisive mathematical development described above emerged in less than two hours of interaction with the AI system.

The mathematical claims and responsibility for the contents of the paper remain entirely with the authors. At the time of posting this preliminary version, the overall argument and its main calculations have been examined by the authors, but a complete independent line-by-line human verification has not yet been carried out.

\section{The two-step elliptic model}\label{sec:step}

\subsection{Potential, speeds, and interface laws}
For the moment let $0<a<b$, $\nu>1$, and $0<h<E<1$ be arbitrary.
The discontinuous primitive is
\begin{equation}\label{eq:stepM}
 M_0(s)=
 \begin{cases}
 0,&0\le s<a^2,\\
 h,&a^2<s<b^2,\\
 1,&s>b^2.
 \end{cases}
\end{equation}
We work on the energy level $\cH=E/2$.  The speeds in the two accessible
regions are
\begin{equation}\label{eq:speeds}
 u=\sqrt E,\qquad w=\sqrt{E-h},\qquad
 q=\frac wu,\qquad \alpha=\frac ab.
\end{equation}
The values at the two interfaces are immaterial to the scattering rules.
Across a potential step, the tangential velocity is conserved and the
normal velocity is determined by energy.  If the jump of $M_0$ in the
outward direction is $J>0$, outward transmission obeys
\[
 (v_N^+)^2=(v_N^-)^2-J.
\]
It is allowed only when the right-hand side is positive.  The outer
interface is reflecting because $E<1$.  At the inner interface,
trajectories with insufficient normal kinetic energy may undergo total
internal reflection; we will construct an orbit uniformly separated from
that threshold.

For calculations of the collision map it is convenient to normalize both
length and inner speed.  Set
\begin{equation}\label{eq:normalization}
 \widehat R=\frac{R}{b},\qquad
 \tau=\frac{u}{b}t.
\end{equation}
Since $u^2=E$, division of the energy identity by $E$ gives
\[
 \frac12\left|\frac{d\widehat R}{d\tau}\right|^2
 +\frac12\widehat M_0\bigl(\Psi_\nu(\widehat R)\bigr)=\frac12,
 \qquad
 \widehat M_0(s):=\frac{M_0(b^2s)}{E}.
\]
Thus, in normalized variables,
\[
 \widehat b=1,\qquad \widehat a=\alpha,\qquad
 \widehat u=1,\qquad \widehat w=q,
\]
whereas the three plateau heights of the primitive become
\[
 0,\qquad \frac hE=1-q^2,\qquad \frac1E.
\]
In particular the normalized total energy is $1$, while the outer plateau
is $1/E>1$ because $E<1$; the outer reflection therefore keeps a strict
energy margin.  The collision laws depend only on the interface geometry
and on the speed ratio $q$, so from now on we suppress the hats in the
kinematic calculations and write $b=u=1$, $a=\alpha$.  The physical
normalization \eqref{eq:physicalfamily} will be restored before smoothing
the potential.

\subsection{The reduced incoming section}
An inner-boundary point is parametrized by its eccentric angle:
\begin{equation}\label{eq:ellipseparam}
 r_\eps(\theta)=a\left(\cos\theta,
              \frac{\sin\theta}{\sqrt{1+\eps}}\right).
\end{equation}
Let $\nn_\eps(\theta)$ be the outward unit normal and
$\ttt_\eps(\theta)$ the counterclockwise unit tangent.  We take the section
immediately before transmission from the annulus into the inner ellipse.
For later calculations the frame is explicitly
\begin{equation}\label{eq:exactframe}
 \nn_\eps(\theta)
 =\frac{(\cos\theta,\sqrt{1+\eps}\sin\theta)}
        {\sqrt{1+\eps\sin^2\theta}},
 \qquad
 \ttt_\eps(\theta)
 =\frac{(-\sqrt{1+\eps}\sin\theta,\cos\theta)}
        {\sqrt{1+\eps\sin^2\theta}}.
\end{equation}
The coordinates are $(\theta,p)$, where
\begin{equation}\label{eq:incoming}
 V_{\rm ann}=q\bigl(-\sqrt{1-p^2}\,\nn_\eps+p\,\ttt_\eps\bigr),
 \qquad |p|<1.
\end{equation}
The antipodal symmetry $(R,V)\mapsto(-R,-V)$ identifies
$\theta$ with $\theta+\pi$ and leaves $p$ unchanged.  We henceforth use
this reduced section.  The axial orbits become fixed points
\begin{equation}\label{eq:XY}
 X=(0,0),\qquad Y=(\pi/2,0).
\end{equation}
In the unreduced section they are two-cycles.

The antipodal identification is used only to turn the two axial cycles
into fixed points and to remove the almost diametrical jump from the
angular coordinate. It does not identify the two different axes. A lift
of the long-axis state has $\theta=0$ or $\pi$, whereas a lift of the
short-axis state has $\theta=\pi/2$ or $3\pi/2$.

One return on the branch considered here consists of an inner entry
transmission, one chord of the inner ellipse, an inner exit transmission,
an annular flight, one outer reflection, and the return annular flight.
The final state is taken before the next inner entry transmission.

\begin{remark}[Section coordinates]
For $\eps>0$, eccentric angle and normalized tangential velocity are not
canonical area coordinates.  The invariant section area has the density
$|r_\eps'(\theta)|$ up to a constant.  The determinant-one calculation
below is made for the limiting circular map, where this density is
constant; we do not assume that the exact elliptic map preserves
$d\theta\dd p$.
\end{remark}

\section{The nearly circular limit map}\label{sec:limit}

Throughout this section $\alpha\in(0,1)$ and $\gamma>0$ are fixed, and
\begin{equation}\label{eq:singularscaling}
 \nu=1+\eps,\qquad q=\gamma\eps,\qquad \eps\downarrow0.
\end{equation}
The annular speed tends to zero, but the collision map is expressed in
normalized directions rather than in flight times.

\paragraph{Why a nontrivial limit is possible.}
At $\eps=0$ the inner transit is diametrical. A small positive $\eps$
changes the exit tangential velocity by an amount of order $\eps$.
The momentum used in the annulus is that velocity divided by
$q=\gamma\eps$, so the normalized change is of order one. Thus the
relevant limit is not the collision map of a circular table with a fixed
speed ratio: the geometric asymmetry and the annular speed vanish together.

\subsection{The inner transit}
After the entry refraction, the inner velocity is
\begin{equation}\label{eq:innerV}
 V=-\sqrt{1-q^2p^2}\,\nn_\eps(\theta)+qp\,\ttt_\eps(\theta).
\end{equation}
The second endpoint of the inner chord is given exactly by
\begin{equation}\label{eq:chordexact}
 r_*=r_\eps(\theta)+\tau V,\qquad
 \tau=-\frac{2r_\eps(\theta)^TD_{1+\eps}V}
                     {V^TD_{1+\eps}V}.
\end{equation}
Let $\theta_*$ be a local lift of its eccentric angle near
$\theta+\pi$.

\begin{lemma}[First-order inner scattering]\label{lem:innerscattering}
For fixed $\gamma$, as $\eps\to0$,
\begin{align}
 \theta_*&=\theta+\pi+
   \eps\bigl(\sin(2\theta)-2\gamma p\bigr)+O(\eps^2),
       \label{eq:thetastar}\\
 V\cdot\ttt_\eps(\theta_*)&=
   \eps\bigl(\gamma p-\sin(2\theta)\bigr)+O(\eps^2).
       \label{eq:vtstar}
\end{align}
The remainders are uniform together with derivatives of any fixed order
on compact subsets of the incoming coordinate domain.
\end{lemma}

\begin{proof}
Write $c=\cos\theta$, $s=\sin\theta$,
$\er=(c,s)$, and $\et=(-s,c)$.  Direct normalization of
$D_{1+\eps}r_\eps$ gives
\[
 \nn_\eps=\er+\frac\eps2 sc\,\et+O(\eps^2),
 \qquad
 \ttt_\eps=\et-\frac\eps2 sc\,\er+O(\eps^2).
\]
For reference, the two scalar products in the chord formula are
\[
 r_\eps^TD_{1+\eps}V
       =-a(1+\eps s^2/2)+O(\eps^2),
 \qquad
 V^TD_{1+\eps}V=1+\eps s^2+O(\eps^2).
\]
Consequently
\begin{align*}
 a^{-1}r_\eps&=(1-\eps s^2/2)\er-\eps sc\,\et/2+O(\eps^2),\\
 V&=-\er+\eps(\gamma p-sc/2)\et+O(\eps^2),\\
 \tau&=2a(1-\eps s^2/2)+O(\eps^2).
\end{align*}
Substituting into \eqref{eq:chordexact},
\[
 a^{-1}r_*=
 -(1-\eps s^2/2)\er+
 \eps(2\gamma p-3sc/2)\et+O(\eps^2).
\]
On the other hand,
\[
 a^{-1}r_\eps(\theta+\pi+\eps\zeta)
 =-(1-\eps s^2/2)\er+
        \eps(sc/2-\zeta)\et+O(\eps^2).
\]
Comparison of the tangential coefficients yields
$\zeta=2sc-2\gamma p$, which is \eqref{eq:thetastar}.
Moreover,
\[
 \ttt_\eps(\theta_*)=
 -\et+\eps(\zeta+sc/2)\er+O(\eps^2).
\]
Taking its scalar product with $V$ gives
$-\eps(\zeta+\gamma p)+O(\eps^2)$, proving
\eqref{eq:vtstar}.  All expressions used in these expansions are smooth
in $(\eps,\theta,p)$ near $\eps=0$ on the indicated compacts, with
nonvanishing denominators.  Taylor's formula gives the stated derivative
control.
\end{proof}

Tangential velocity is conserved in the exit refraction.  Its normalized
annular value is therefore
\begin{equation}\label{eq:wexact}
 w_\eps(\theta,p)=\frac{V\cdot\ttt_\eps(\theta_*)}{\gamma\eps}
   =p-\gamma^{-1}\sin(2\theta)+O(\eps).
\end{equation}
The outgoing annular direction is
\[
 \sqrt{1-w_\eps^2}\,\nn_\eps(\theta_*)
       +w_\eps\,\ttt_\eps(\theta_*),
\]
provided $|w_\eps|<1$.

\subsection{The annular excursion}
For concentric circles, let the initial eccentric (now also polar)
angle be $\theta$, and let $w$ be the outgoing tangential fraction at the
inner circle. Normalize the speed on this excursion to one. The conserved
angular momentum is $aw$. The velocity makes an angle $\arcsin w$ with
the outward radius initially, and an angle $\arcsin(\alpha w)$ with the
outward radius when it reaches the outer circle. Since the velocity
direction is constant during the straight flight, the change in polar
angle is their difference.

At the outer reflection the radial velocity reverses and angular momentum
is unchanged. The return leg contributes the same angular increment.
At the inner return, the tangential fraction is again $w$. Consequently
\begin{equation}\label{eq:annulusmap}
 (\theta,w)\longmapsto(\theta+g_\alpha(w),w),
 \qquad
 g_\alpha(w)=2\bigl(\arcsin w-\arcsin(\alpha w)\bigr).
\end{equation}
The angles are signed; the same formula applies when $w<0$.

Combining \eqref{eq:wexact} and \eqref{eq:annulusmap}, and removing the
antipodal increment $\pi$, gives the limiting map
\begin{equation}\label{eq:limitmap}
 \boxed{
 \begin{aligned}
 p_+&=p-\gamma^{-1}\sin(2\theta),\\
 \theta_+&=\theta+g_\alpha(p_+)\pmod\pi.
 \end{aligned}}
\end{equation}
It is considered only on
\begin{equation}\label{eq:regulardomain}
 \mathcal D_{\alpha,\gamma}
 =\{(\theta,p): |p|<1,\ |p-\gamma^{-1}\sin(2\theta)|<1\}.
\end{equation}
No definition is imposed here beyond the transmission threshold.

The domain $\mathcal D_{\alpha,\gamma}$ is a branch domain, not an
assertion that every initial state stays on this branch forever. Only the
orbit constructed below, and sufficiently small neighbourhoods of the
pieces used in its continuation, will be required to remain there.
The map is a local diffeomorphism: the drift can be inverted first, and
then the kick.  In these limiting coordinates it has determinant one.

\begin{proposition}[Differentiable convergence of the exact map]\label{prop:C1}
Let $F_\eps$ denote the exact reduced elliptic return map on the stated
itinerary, and let $F_0$ be \eqref{eq:limitmap}.  On every compact subset
of \eqref{eq:regulardomain}, $F_\eps$ is well defined for all sufficiently
small positive $\eps$ and
\begin{equation}\label{eq:mapconvergence}
 F_\eps=F_0+O_{C^r}(\eps)
\end{equation}
for every fixed finite $r$, with local lifts for the angular coordinate.
In particular the convergence is $C^1$.
\end{proposition}

\begin{proof}
The inner chord and its endpoint are smooth by \eqref{eq:chordexact}.
The numerator in \eqref{eq:wexact} vanishes identically at $\eps=0$
for every $(\theta,p)$.  If it is denoted by $f(\eps,\theta,p)$, then
\[
 \frac{f(\eps,\theta,p)}\eps
 =\int_0^1\partial_\eps f(t\eps,\theta,p)\dd t.
\]
Thus the apparent singularity in \eqref{eq:wexact} is removable,
including derivatives in the section variables.  This argument fixes
$\gamma$; it gives no uniform estimate as $\gamma\to0$.
The quotient is therefore the restriction of a smooth function of
$(\eps,\theta,p)$ through $\eps=0$, not merely a bounded quotient.
This distinction is essential for the later use of transverse persistence.

On a compact subset of \eqref{eq:regulardomain}, the square root of
$1-w_\eps^2$ is uniformly regular.  In the annulus use the direction
$V/q$, not the small speed $q$, to compute the intersections.  A line
$R+tU$ meets a level ellipse $\Psinu=c$ by solving
\[
 (U^T\Dnu U)t^2+2(R^T\Dnu U)t+\Psinu(R)-c=0.
\]
At $\eps=0$ the selected intersections with the outer circle and the
return inner circle are transverse.  Their positive roots, the specular
reflection, and the final section coordinates therefore depend smoothly
on $(\eps,\theta,p)$ near the compact set.  The root corresponding to
first inward return remains the selected root by continuity. More
explicitly, at $\eps=0$ the outer circle is reached with strictly positive
radial velocity and the inner circle is reached with strictly negative
radial velocity. Both signs have uniform margins on the compact set. The
chosen roots are simple and stay separated from the other roots, so no
extra collision or grazing event can enter the prescribed itinerary.  Composing
these smooth maps proves \eqref{eq:mapconvergence}.
\end{proof}

\section{An explicit heteroclinic of the limit map}\label{sec:contraction}

From now on
\begin{equation}\label{eq:alphagamma}
 \alpha=\frac12,\qquad \gamma=\frac1{16},
 \qquad g=g_{1/2}.
\end{equation}
Since $g'(0)=1$, the derivative matrices of $F_0$ at the axial fixed
points are
\begin{equation}\label{eq:matrices}
 DF_0(X)=\begin{pmatrix}-31&1\\-32&1\end{pmatrix},
 \qquad
 DF_0(Y)=\begin{pmatrix}33&1\\32&1\end{pmatrix}.
\end{equation}
Their traces are $-30$ and $34$, respectively, and their determinants are
one.  Both points are hyperbolic saddles.

The choices $\alpha=1/2$ and $\gamma=1/16$ have separate purposes.
The annular drift can then span any angle in $(-2\pi/3,2\pi/3)$, so a
transition of size $\pi/2$ leaves room for small corrections. The kick
strength $\gamma^{-1}=16$ makes those corrections contractive. The proof
will use both margins explicitly; the numerical values are convenient
choices rather than fitted parameters.

\subsection{Inverting the drift}
For $0<\alpha<1$, the map $g_\alpha$ is strictly increasing.  It maps
$(-1,1)$ onto the interval $(-2\arccos\alpha,2\arccos\alpha)$, and its inverse is
\begin{equation}\label{eq:inversedrift}
 P_\alpha(d)=\frac{\sin(d/2)}{
       \sqrt{1+\alpha^2-2\alpha\cos(d/2)}}.
\end{equation}
To check the inverse without squaring away its sign, put $\beta=d/2$
and $L=(1+\alpha^2-2\alpha\cos\beta)^{1/2}$. Since
$|\beta|<\arccos\alpha$, we have $\cos\beta>\alpha$ and
\[
 p=\frac{\sin\beta}{L},\quad
 \sqrt{1-p^2}=\frac{\cos\beta-\alpha}{L},\quad
 \sqrt{1-\alpha^2p^2}=\frac{1-\alpha\cos\beta}{L}.
\]
The sine and cosine of $\arcsin p-\arcsin(\alpha p)$ are therefore
$\sin\beta$ and $\cos\beta$, with the prescribed angle branch. This
proves \eqref{eq:inversedrift}. Moreover,
\[
 g_\alpha'(p)=2\left(\frac1{\sqrt{1-p^2}}
       -\frac\alpha{\sqrt{1-\alpha^2p^2}}\right)
       \ge2(1-\alpha),
\]
where the lower bound follows by first replacing
$(1-p^2)^{-1/2}$ by the smaller quantity
$(1-\alpha^2p^2)^{-1/2}$. Thus
\begin{equation}\label{eq:Pbound}
 |P_\alpha(d)|<1,
 \qquad 0<P_\alpha'(d)\le\frac1{2(1-\alpha)}.
\end{equation}
Write $P=P_{1/2}$.  Its domain is $(-2\pi/3,2\pi/3)$ and $|P'|\le1$.

For a lifted orbit whose increments lie in this interval, the map
\eqref{eq:limitmap} is equivalent to
\begin{equation}\label{eq:recurrence}
 \sin(2\theta_n)=\frac1{16}
 \bigl[P(\theta_n-\theta_{n-1})
             -P(\theta_{n+1}-\theta_n)\bigr],
 \qquad n\in\Z,
\end{equation}
with
\begin{equation}\label{eq:momentfromangle}
 p_n=P(\theta_n-\theta_{n-1}).
\end{equation}

\subsection{A contraction on bi-infinite sequences}
Let
\begin{equation}\label{eq:stepsequence}
 s_n=\begin{cases}0,&n\le0,\\\pi/2,&n\ge1,\end{cases}
 \qquad \sigma_n=\cos(2s_n)\in\{1,-1\}.
\end{equation}
The reference sequence $s$ records the intended itinerary, not a solution:
it stays exactly on the long-axis symbol for $n\le0$ and then makes a
single jump of size $\pi/2$ to the short-axis symbol for $n\ge1$.  Thus the
large angular transition is already present in $s$; the corrections $u_n$
only adjust this ideal itinerary so that the exact second-order recurrence
is satisfied.  The Banach space below contains genuine bi-infinite
sequences; no finite truncation is part of the existence argument.

We seek $\theta_n=s_n+u_n$ with
\[
 u\in\cB:=\{u\in\ell^\infty(\Z):\|u\|_\infty\le1/8\}.
\]
For every $u\in\cB$,
\begin{equation}\label{eq:incrementmargin}
 |\theta_{n+1}-\theta_n|
 \le d_*:=\frac\pi2+\frac14<\frac{2\pi}{3}.
\end{equation}
All terms in \eqref{eq:recurrence} are therefore defined, with a uniform
margin from the endpoints of the inverse-drift domain.

Define $\mathcal T:\cB\to\ell^\infty(\Z)$ by
\begin{equation}\label{eq:contractionoperator}
 (\mathcal T u)_n=
 \frac{\sigma_n}{2}\arcsin\left(
 \frac{P(\theta_n-\theta_{n-1})
              -P(\theta_{n+1}-\theta_n)}{16}\right).
\end{equation}
The relation $u=\mathcal T u$ is equivalent to
\eqref{eq:recurrence} within $\cB$, because $2|u_n|\le1/4<\pi/2$.

\begin{lemma}[Uniform contraction]\label{lem:contraction}
The map $\mathcal T$ takes $\cB$ strictly into itself and
\begin{equation}\label{eq:kappa}
 \|\mathcal T u-\mathcal T v\|_\infty
 \le\kappa\|u-v\|_\infty,
 \qquad \kappa=\frac1{\sqrt{63}}<1.
\end{equation}
Consequently it has a unique fixed point in $\cB$.
The uniqueness is within this ball and for this fixed placement of the
transition; it is not a claim that the full collision map has only one
heteroclinic orbit.
\end{lemma}

\begin{proof}
By \eqref{eq:Pbound}, the argument of the arcsine has absolute value at
most $1/8$.  Hence
\[
 \|\mathcal T u\|_\infty\le\tfrac12\arcsin(1/8)<1/8.
\]
The difference of the two $P$ terms in \eqref{eq:contractionoperator}
is Lipschitz in $u$ with constant at most $4$.  On $[-1/8,1/8]$ the
arcsine derivative is bounded by $8/\sqrt{63}$.  Thus the total constant
is
\[
 \frac12\,\frac8{\sqrt{63}}\,\frac1{16}\,4
       =\frac1{\sqrt{63}}.
\]
The contraction principle on the complete closed ball $\cB$ concludes
the proof.
\end{proof}

\subsection{Exponential tails and a uniform transmission margin}
Let $u$ be the fixed point, and let $u^{(m)}=\mathcal T^m(0)$.
The first iterate is supported on $\{0,1\}$, and the operator is
nearest-neighbour in the sequence index.  Inductively,
\begin{equation}\label{eq:finitesupport}
 \supp u^{(m)}\subseteq\{1-m,\ldots,m\},\qquad m\ge1.
\end{equation}
At the same time,
\begin{equation}\label{eq:picarderror}
 \|u-u^{(m)}\|_\infty\le\kappa^m\|u\|_\infty.
\end{equation}
For an index at distance $d\ge1$ from $\{0,1\}$, choosing $m=d$
in \eqref{eq:finitesupport}--\eqref{eq:picarderror} gives
$|u_n|\le\|u\|_\infty\kappa^d$.  Formula
\eqref{eq:momentfromangle}, together with $P(0)=0$ and $|P'|\le1$,
gives the same type of decay for $p_n$ away from the central jump.
Therefore
\begin{equation}\label{eq:maptails}
 (\theta_n,p_n)\longrightarrow X\quad(n\to-\infty),
 \qquad
 (\theta_n,p_n)\longrightarrow Y\quad(n\to+\infty),
\end{equation}
exponentially in the iteration index.

Furthermore, \eqref{eq:incrementmargin} implies the exact bound
\begin{equation}\label{eq:rhomargin}
 |p_n|\le\rho:=P(d_*)<1\qquad(n\in\Z).
\end{equation}
The normalized outgoing momentum at the $n$th inner transit is
$p_{n+1}$, so the same bound holds for it.  Thus the closure of the orbit,
including $X$ and $Y$, is compactly contained in the regular transmission
branch.  The decimal value $\rho\simeq0.98984214$ is only an illustration;
the strict inequality in \eqref{eq:rhomargin} follows from
$d_*<2\pi/3$.

\subsection{Why the same contraction gives transversality}
The existence of a connecting sequence is not yet its robustness. To
obtain persistence we must rule out a common tangent to the stable and
unstable curves. The useful observation is that such a tangent would
produce a bounded solution of the linearized recurrence on all of $\Z$.
We first exclude that possibility directly.

\begin{lemma}[No nonzero bounded variational sequence]\label{lem:nobounded}
Along the orbit constructed above, the only bounded solution
$(v_n,\pi_n)_{n\in\Z}$ of the linearized map is the zero solution.
\end{lemma}
\begin{proof}
Differentiate the inverse-drift relation:
\begin{equation}\label{eq:variationalmoment}
 \pi_n=P'(\theta_n-\theta_{n-1})(v_n-v_{n-1}).
\end{equation}
The differentiated angular recurrence is
\begin{equation}\label{eq:linearrecurrence}
\begin{split}
 2\cos(2\theta_n)v_n=\frac1{16}\bigl[&
 P'(\theta_n-\theta_{n-1})(v_n-v_{n-1})\\
 &-P'(\theta_{n+1}-\theta_n)(v_{n+1}-v_n)\bigr].
\end{split}
\end{equation}
At the fixed point, the arcsine argument in
\eqref{eq:contractionoperator} has modulus at most $1/8$. Hence
\[
 |\cos(2\theta_n)|=\cos(2u_n)\ge\frac{\sqrt{63}}8.
\]
Using $|P'|\le1$ in \eqref{eq:linearrecurrence} gives, for each $n$,
\[
 \frac{\sqrt{63}}4|v_n|
       \le\frac4{16}\|v\|_\infty.
\]
Taking the supremum, which need not be attained, yields
$\|v\|_\infty\le\|v\|_\infty/\sqrt{63}$. Thus $v=0$.
Equation~\eqref{eq:variationalmoment} then implies $\pi=0$.
\end{proof}

\begin{proposition}[Transverse limit connection]\label{prop:limitconnection}
The orbit constructed above is a transverse heteroclinic connection from
$X$ to $Y$ for $F_0$.
\end{proposition}
\begin{proof}
Existence and the limiting fixed points follow from
\eqref{eq:maptails}. Suppose the two invariant curves at an orbit point
had the same tangent line, and choose a nonzero vector on that line.
Iterating the derivative forward produces a tangent to $W^s(Y)$;
iterating it backward produces a tangent to $W^u(X)$. Once the orbit
enters the respective local hyperbolic neighbourhoods, these tangent
vectors decay geometrically in the corresponding time direction. The
finitely many remaining iterates have bounded derivatives. We obtain a
nonzero variational sequence bounded on $\Z$, contradicting
Lemma~\ref{lem:nobounded}.
\end{proof}

Thus the contraction does two jobs: it constructs the orbit and excludes
a tangent perturbation that decays at both ends. No numerical angle
between the invariant curves is used.

\section{Persistence for the elliptic step model}\label{sec:ellipticpersistence}

We will use the following standard consequence of the local stable
manifold theorem. We state the exact version needed, since it explains
why the infinitely many returns of a heteroclinic do not require summing
infinitely many perturbation errors. The standard input is the $C^1$
continuous dependence of local stable and unstable curves of a hyperbolic
fixed point on the map; see, for example, \cite{HPS}.

\begin{lemma}[Persistence from a finite connecting segment]\label{lem:persistence}
Let $f_0$ be a smooth local diffeomorphism of a surface, with hyperbolic
saddle fixed points $X,Y$ and a transverse heteroclinic orbit between them.
Suppose the orbit together with its two limits has compact closure inside
the regular domain of the map. If smooth local diffeomorphisms $f_s$
converge to $f_0$ in $C^1$ on a neighbourhood of that closure, then, for
small $s$, the continued saddles have a nearby transverse heteroclinic
connection. Its tails converge exponentially to the saddles.
\end{lemma}
\begin{proof}
The saddles persist because $Df_0-I$ is invertible at each of them.
The local stable manifold theorem gives local invariant curves, with
uniform local contraction/expansion and $C^1$ continuous dependence on
$f_s$.

Fix a connecting point $z_0$. Choose $m_-,m_+<\infty$ so that
$f_0^{-m_-}(z_0)$ belongs to a small local unstable segment at $X$, and
$f_0^{m_+}(z_0)$ belongs to a small local stable segment at $Y$. Take
segments on which these two membership statements are interior. Their
images under $f_s^{m_-}$ and $f_s^{-m_+}$ are $C^1$ close to the two
original curves near $z_0$. Only these finitely many iterates are being
compared. The required local inverse branches also converge, by
invertibility and $C^1$ convergence.

Choose the vertical coordinate direction not parallel to either tangent line at
$z_0$. In the resulting chart the two curves can be written as
$y=u_s(x)$ and $y=v_s(x)$. At the limiting intersection,
$u_0-v_0$ has a simple zero. On a small fixed interval its derivative
has a fixed nonzero sign and its endpoint values have opposite signs.
Both properties persist under $C^1$ convergence. There is therefore a
unique nearby zero of $u_s-v_s$, and the intersection is transverse.

The intersection belongs to the global unstable manifold of the first
saddle and the global stable manifold of the second. Its infinite tails
are supplied by those invariant manifolds and stay in their local regular
neighbourhoods after finitely many iterates. This proves the assertion.
\end{proof}

\begin{proposition}[Elliptic connection]\label{prop:ellipticconnection}
For all sufficiently small $\eps>0$, the two-step elliptic model with
\begin{equation}\label{eq:epsstepfamily}
 \nu=1+\eps,\qquad \alpha=\frac12,\qquad q=\frac\eps{16}
\end{equation}
has a transverse heteroclinic connection from the long-axis periodic
orbit to the short-axis periodic orbit.  Its reduced return orbit has
exponential tails and stays in the branch with two inner transmissions
and one outer reflection per return, uniformly away from the singular
boundaries for this fixed $\eps$.
\end{proposition}

\begin{proof}
By \eqref{eq:rhomargin}, the orbit in
Proposition~\ref{prop:limitconnection}, together with $X$ and $Y$, is a
compact subset of \eqref{eq:regulardomain}.  Choose a slightly larger
compact neighbourhood still contained in that domain.  Proposition~\ref{prop:C1}
provides $C^1$ convergence there.  The axial states $X$ and $Y$ are fixed
points of the reduced exact map for every positive $\eps$, by the exact
coordinate symmetries.  Their derivative matrices converge to
\eqref{eq:matrices}, so they remain hyperbolic.

Lemma~\ref{lem:persistence} gives a transverse connection
near the limiting one.  Shrinking the neighbourhoods and $\eps$ if
necessary keeps its finite connecting portion on the regular branch.
Its tails stay in small regular neighbourhoods of the two fixed points.
This also supplies a uniform positive margin from critical transmission
and tangency for the closure of the resulting orbit.

Lifting from the antipodal quotient gives a connection between the
corresponding axial two-cycles of the original billiard.  This is the
claimed connection between the two periodic trajectories.
\end{proof}

The physical parameters \eqref{eq:physicalfamily} realize
\eqref{eq:epsstepfamily}, since
\[
 \sqrt E=\frac1{\sqrt2},\qquad
 \sqrt{E-h}=\frac\eps{16\sqrt2}.
\]
The proposition does not quantify the admissible threshold $\eps_0$.
In particular, a floating-point computation at $\eps=0.01$ is not a
certificate that this particular value lies below the threshold.

\begin{remark}[The time scale]
For each fixed positive $\eps$, the return times near the axial cycles
are finite and positive.  Thus exponential decay in the number of
returns implies exponential decay in physical time.  No positive decay
rate uniform as $\eps\to0$ is claimed: the annular speed tends to zero.
\end{remark}

\section{Differentiable regularization of a single layer}\label{sec:layers}

We now fix the elliptic and energetic parameters.  The smoothing argument
is independent of the construction of the step-model connection and is
stated for general $\nu>1$, $0<a<b$, and $0<h<E<1$.
The differentiable approximation of reflecting walls has an established
soft-billiard framework; see \cite{RRT}.  Since the present return map also
contains two transmissions, we prove the necessary layer statement
directly rather than importing a theorem for reflecting walls alone.

\subsection{An extended fast-time system}
Consider one level interface $\Psinu(R)=c$, where $c>0$, and a positive
jump $J$ of the primitive.  In the layer let
\begin{equation}\label{eq:singleprofile}
 M(s)=M_-+J\chi\left(\frac{s-c}{\delta}\right),
 \qquad \delta>0.
\end{equation}
The physical force in this layer is
$-J\chi'((\Psinu-c)/\delta)\nabla\Psinu/(2\delta)$.
Introduce
\[
 z=\frac{\Psinu(R)-c}{\delta},
 \qquad \tau=\frac{t-t_{\rm in}}\delta.
\]
With dots denoting derivatives in $\tau$, the equations become
\begin{equation}\label{eq:fastsystem}
 \boxed{
 \begin{aligned}
 \dot R&=\delta V,\\
 \dot V&=-\frac J2\chi'(z)\nabla\Psinu(R),\\
 \dot z&=\nabla\Psinu(R)\cdot V.
 \end{aligned}}
\end{equation}
Regard $(R,V,z)$ as independent coordinates of an extended system.  Its
vector field is smooth also at $\delta=0$.  The constraint
\begin{equation}\label{eq:constraint}
 \Psinu(R)-c-\delta z=0
\end{equation}
is invariant, since its derivative along \eqref{eq:fastsystem} vanishes.
Thus compatible initial data reproduce the physical system exactly for
$\delta>0$.  The large physical derivatives have been removed by a change
of variables, not estimated as a small perturbation.

For reference, the constraint calculation is
\[
 \frac{d}{d\tau}\bigl(\Psinu(R)-c-\delta z\bigr)
 =\delta\nabla\Psinu(R)\cdot V
       -\delta\nabla\Psinu(R)\cdot V=0.
\]
The extended energy $|V|^2+J\chi(z)$ is conserved as well. At $\delta=0$
the variable $z$ continues to evolve even though $R$ is frozen: it records
the progress through the collapsed layer. Treating it as an independent
variable is what makes the limiting problem regular.

\subsection{The limiting scattering law}
At $\delta=0$, the position is fixed, say $R=R_0$.  Write
\[
 G=|\nabla\Psinu(R_0)|>0,
 \qquad \nn=\frac{\nabla\Psinu(R_0)}G,
 \qquad w=V\cdot\nn.
\]
The tangential component $V_T$ is constant, and
\begin{equation}\label{eq:normalfast}
 \dot z=Gw,\qquad \dot w=-\frac{JG}{2}\chi'(z),
 \qquad w^2+J\chi(z)=\text{constant}.
\end{equation}
It follows that the limiting scattering laws are:
\begin{align}
 w_{\rm out}&=-\sqrt{w_{\rm in}^2+J}
 &&\text{for entry from $z=+1$ to $z=-1$, }w_{\rm in}<0,
 \label{eq:downward}\\
 w_{\rm out}&=\sqrt{w_{\rm in}^2-J}
 &&\text{for transmission from $z=-1$ to $z=+1$, }w_{\rm in}^2>J,
 \label{eq:upward}\\
 w_{\rm out}&=-w_{\rm in}
 &&\text{for reflection from $z=-1$ back to $z=-1$, }0<w_{\rm in}^2<J.
 \label{eq:reflection}
\end{align}
In \eqref{eq:upward} and \eqref{eq:reflection} the incident normal
velocity is positive.  In every case $V_T$ is unchanged.  Thus these are
exactly the refraction and specular-reflection laws of the step model.

For reflection, the turning point $z_*$ is determined by
\begin{equation}\label{eq:turning}
 \chi(z_*)=w_{\rm in}^2/J,\qquad -1<z_*<1.
\end{equation}
It is nondegenerate because $\chi'(z_*)>0$: at the turning point,
$\dot w=-JG\chi'(z_*)/2<0$, so the normal velocity actually changes sign.
The limiting round-trip time is
\begin{equation}\label{eq:reflectiontime}
 T_0=\frac2G\int_{-1}^{z_*}
       \frac{\dd z}{\sqrt{w_{\rm in}^2-J\chi(z)}}.
\end{equation}
The square-root singularity at $z_*$ is integrable.  On a compact family
of incident states separated from $w_{
\rm in}=0$ and $w_{\rm in}^2=J$, the turning points remain in a compact
subset of $(-1,1)$.

\begin{lemma}[A smooth layer approximates its ideal scattering map]\label{lem:layer}
Fix a compact family of incident states for one of
\eqref{eq:downward}--\eqref{eq:reflection}.  Assume that all incident and
outgoing normal velocities are nonzero and that the relevant strict
transmission or reflection inequalities have positive margins.
Identify the layer faces $\Psinu=c\pm\delta$ smoothly with
$\Psinu=c$.  Then the map $\mathcal L_\delta$ from entry to exit satisfies
\begin{equation}\label{eq:layerestimate}
 \mathcal L_\delta=\mathcal L_0+O_{C^r}(\delta)
\end{equation}
for every fixed finite $r$.  Here $\mathcal L_0$ is the corresponding
ideal scattering law.  The physical traversal time, including its
section-variable derivatives of each fixed order, is $O(\delta)$.
\end{lemma}

\begin{proof}
\emph{Finite limiting passages.}
For transmission, the normal energy in \eqref{eq:normalfast} bounds
$|w|$ away from zero throughout the passage; the limiting exit time is
finite. For reflection, finiteness follows from
\eqref{eq:reflectiontime}: near $z_*$ the radicand is comparable to
$z_*-z$, whose inverse square root is integrable. The lower bound on
incident normal speed and the upper margin from $w_{\rm in}^2=J$ keep
$z_*$ in a compact subinterval of $(-1,1)$.

\emph{Smooth exit time.}
Parametrize the incident positions by
\[
 R_{\rm in}(\theta,\delta)=
 \sqrt{c+\delta z_{\rm in}}
       \left(\cos\theta,\frac{\sin\theta}{\sqrt\nu}\right),
 \qquad z_{\rm in}\in\{-1,1\}.
\]
Together with regular velocity coordinates, this gives incident data
smooth in the section variables and in $\delta$. The solution of the
extended system depends smoothly on these data near any of the finite
limiting passages. At exit,
\[
 \dot z(T_0)=G w_{\rm out}\ne0.
\]
The equation $z(\tau)=z_{\rm out}$ therefore determines an exit time
$T_\delta$ smooth in the data and in $\delta$, by the implicit function
theorem. In reflection we use the positive root near the limiting
round-trip time $T_0$, not the root at the initial time.

The selected crossing is the physical exit. For transmission the limiting
path is strictly monotone in $z$, with transverse entrance and exit. For
reflection it stays strictly inside the layer between entrance and the
positive return, has one nondegenerate turning point, and is transverse
at both endpoints. On compact time intervals away from these events there
is a strict gap from the exit face. These gaps and signs persist for small
$\delta$, excluding an earlier unwanted exit. A finite cover of the
compact family of incident states makes the choices uniform.

\emph{The derivative is a section derivative.}
Let $Z=(R,V,z)$, let $\mathcal X_\delta$ be the extended vector field,
and let $Y$ be the derivative of the fixed-time solution with respect to
an incident parameter $a$. Differentiating
$z(T_\delta(a),a)=z_{\rm out}$ gives
\begin{equation}\label{eq:hittingderivative}
 D_aT_\delta=-\frac{Y_z}{(\mathcal X_\delta)_z},
 \qquad
 D_aZ_{\rm exit}
 =Y-\mathcal X_\delta\frac{Y_z}{(\mathcal X_\delta)_z}.
\end{equation}
All quantities are evaluated at exit. The denominator is bounded away
from zero on the compact family. This formula displays the hitting-time
correction that is absent from a derivative of the flow at a fixed time.
Higher derivatives follow from the same smooth implicit function theorem.

Thus the exit state and its section-variable derivatives are smooth in
$\delta$ through zero. Taylor's formula proves
\eqref{eq:layerestimate}, after projecting to $(R,V)$ and making the
smooth face identifications. The physical passage time is
$\delta T_\delta$, hence is $O_{C^r}(\delta)$. In particular, we never
need to differentiate the improper integral
\eqref{eq:reflectiontime} at its moving endpoint.
\end{proof}

\begin{remark}[Nonuniformity]
The constants in Lemma~\ref{lem:layer} depend on the fixed geometry,
speeds, and margins.  In our application they may become very large as
$\eps\to0$, because the annular normal velocities become small.  The
proof never asks for uniformity in that limit.
\end{remark}

\section{From the elliptic connection to a smooth positive coefficient}\label{sec:realization}

\subsection{Smoothing the two steps}
Fix one sufficiently small positive $\eps$ for which
Proposition~\ref{prop:ellipticconnection} holds.  Thus $\nu,a,b,E,h,q$
are now fixed.  Define
\begin{equation}\label{eq:smoothedM}
 M_\delta(s)=h\chi\left(\frac{s-a^2}{\delta}\right)
 +(1-h)\chi\left(\frac{s-b^2}{\delta}\right),
 \qquad m_\delta=M_\delta'.
\end{equation}
Take $\delta$ sufficiently small that both layers are disjoint and
contained in $s>0$.  This potential agrees exactly with the plateau
values away from the layers.  The resulting equation is already a
$C^\infty$ Hamiltonian system, including at the origin, since
$\Psinu$ is a polynomial.  At this stage $m_\delta$ is nonnegative and
vanishes on the plateaux.

Choose a fixed section in the annulus,
\begin{equation}\label{eq:fixedsection}
 \Psinu(R)=c_*,\qquad
 c_*:=\frac{a^2+b^2}{2},\qquad
 \nabla\Psinu(R)\cdot V<0,
\end{equation}
on the energy surface $\cH=E/2$.  For small enough $\delta$, it stays
outside the layers and $M_\delta(c_*)=h$.  Hence the section speed is
exactly $\sqrt{E-h}$, independently of $\delta$.  The step-model return
map on this section is conjugate by a regular free-flight map to the
inner-interface map used earlier.

Along each inward free leg of the step orbit, every intermediate level is
crossed once before the inner interface. The crossing is transverse.
Therefore moving the section changes the return map by a smooth local
conjugacy, preserving hyperbolicity and transversality. In the regularized
problem this fixed section avoids making the comparison itself depend on
the changing thickness of the layers.

\begin{proposition}[Regularization of a regular transverse connection]\label{prop:regularization}
Suppose the two-step elliptic model has a transverse connection between
two hyperbolic axial cycles and that its closure on a regular section
stays on the branch with two inner transmissions and one outer
reflection, with positive margins from tangencies and critical
transmission.  Then the Hamiltonian defined by \eqref{eq:smoothedM}
has a nearby transverse connection between hyperbolic axial periodic
orbits for every sufficiently small positive $\delta$.
\end{proposition}

\begin{proof}
On a compact neighbourhood of the step-model connection, the return map
on \eqref{eq:fixedsection} is a finite composition of free flights,
entry transmission, exit transmission, and outer reflection.  The free
flights between the section and the layer faces depend smoothly on
$\delta$ and on their initial data, because their intersections are
transverse.  Lemma~\ref{lem:layer} applies to each of the three layer
passages.  For the outer passage the jump is $J=1-h$, and
\[
 0<w_{\rm in}^2\le E-h<1-h=J.
\]
The lower bound has a positive margin on the compact regular set.  The
upper strict inequality follows from $E<1$.  For the inner entry, the squared incident normal speed is
\[
 (E-h)(1-p^2)>0.
\]
For the inner exit, if $w$ denotes the outgoing annular tangential
fraction, the squared normal speed just before transmission is
$E-(E-h)w^2$. The transmission margin is exactly
\[
 \bigl[E-(E-h)w^2\bigr]-h=(E-h)(1-w^2)>0.
\]
The compact bounds $|p|,|w|<1$ make both margins uniform for the fixed
$\eps$. At the outer layer, the gap from the upper reflection threshold
is at least $1-E>0$. This verifies each hypothesis of the layer lemma,
not merely the absence of tangencies.

Consequently the reduced maps, or equivalently their unreduced iterates,
satisfy
\begin{equation}\label{eq:deltaC1}
 \mathcal P_{\eps,\delta}
   =\mathcal P_{\eps,0}+O_{C^1}(\delta)
\end{equation}
on the compact neighbourhood needed for persistence.  The hyperbolic
cycles and their local invariant manifolds continue, and the finite-iterate
argument in Section~\ref{sec:ellipticpersistence} preserves the transverse
intersection.  The symmetries of $M_\delta(\Psinu)$ preserve the coordinate
axes.  The nearby periodic orbits are therefore the axial ones.

One can also identify them directly. The outer turning value $s_*$ is
the unique solution of $M_\delta(s_*)=E$ in the outer layer, where
$M_\delta'(s_*)>0$. The axial trajectories turn at
$x=\pm\sqrt{s_*}$ and $y=\pm\sqrt{s_*/\nu}$, respectively. They are
periodic on the prescribed energy surface and give the fixed points of
the reduced return map on the invariant axes. Thus the continuation is
not merely to unspecified periodic orbits lying near the axes.
\end{proof}

\subsection{Removing degeneracy after smoothing}
Fix a positive $\delta$ supplied by
Proposition~\ref{prop:regularization}.  Now set
\begin{equation}\label{eq:eta}
 M_{\delta,\eta}(s)=M_\delta(s)+\eta s,
 \qquad m_{\delta,\eta}(s)=m_\delta(s)+\eta,
 \qquad \eta>0.
\end{equation}
For fixed $\eps,\delta$, this is an ordinary smooth perturbation:
\[
 -\nabla\bigl(\tfrac12 M_{\delta,\eta}(\Psinu)\bigr)
 +\nabla\bigl(\tfrac12 M_\delta(\Psinu)\bigr)
 =-\eta D_\nu R.
\]
The vector fields converge in every $C^r$ norm on compact subsets as
$\eta\to0$.  On the fixed position section \eqref{eq:fixedsection}, the
energy constraint now gives speed
\[
 \sqrt{E-h-\eta c_*},
\]
which remains positive if $\eta<(E-h)/c_*$.  It identifies the perturbed
energy sections smoothly with the original one.  The return maps are
therefore $C^1$ close for sufficiently small $\eta$.

Hyperbolicity, the local invariant manifolds, and their transverse
intersection persist once more.  The axial symmetry remains exact.
The new coefficient satisfies $m_{\delta,\eta}\ge\eta>0$ on all of
$[0,\infty)$; it is constant outside the two compact layer intervals,
so it and all its derivatives are bounded for the fixed parameter choice.
The primitive is \eqref{eq:Mfinal}, with $M(0)=0$.

This step uses the quadratic nature of $\Psinu$ essentially for its
simple regularity statement at the origin.  There is no loss of
smoothness when adding $\eta\Psinu/2$ to the Hamiltonian.

\subsection{Exponential decay in physical time}
There is one final distinction: the contraction and persistence arguments
concern return indices, whereas the road map requires estimates in
physical time. We now make that conversion explicit, with
$(\eps,\delta,\eta)$ fixed.

Let $z_n$ be the section states along the stable tail and let $Y_\eta$
be the axial fixed point in the reduced section. Local hyperbolicity gives
\[
 |z_n-Y_\eta|\le C\rho^n,\qquad 0<\rho<1.
\]
The return-time function is smooth near $Y_\eta$ and has a strictly
positive finite value there. Consequently the successive return times
$t_n$ obey
\[
 0<T_-\le t_{n+1}-t_n\le T_+<\infty
\]
for all sufficiently large $n$. The full flow is smooth with bounded
first derivatives on these compact pieces, so for $t\in[t_n,t_{n+1}]$
the inactive component satisfies
\[
 |x(t)|+|x'(t)|\le C'\rho^n.
\]
Indeed, compare at the same elapsed time with the axial trajectory
starting at the limiting section point; its $x,x'$ coordinates vanish
identically. Since $t-t_0\le(n+1)T_+$, this is an exponential estimate in
$t$. The unstable tail is treated in reversed time and controls $y,y'$.
Antipodal signs disappear in the squared component estimates.

Squaring and, if necessary, reducing the decay exponent gives
\eqref{eq:hca-decay}. Enlarging $B_0$ covers the finite middle part of the
orbit. No choice of a common asymptotic phase for the two different
periodic orbits is required: only their inactive components enter those
estimates. The return-time lower bounds also exclude accumulation of
infinitely many returns in finite physical time.

The orbit is nontrivial because its energy is $1/4$. Time reversal
$(R,V,t)\mapsto(R,-V,-t)$ supplies the reverse connection, preserving
transversality. This completes the proof of Theorem~\ref{thm:main}.

\begin{remark}[Order of choices and scope]
The quantifiers in Theorem~\ref{thm:main} are essential:
\[
 0<\eps<\eps_0,
 \qquad 0<\delta<\delta_0(\eps),
 \qquad 0<\eta<\eta_0(\eps,\delta).
\]
No uniform thickness estimate as $\eps\to0$ is used.  No extra step or
angular correction is added.  After adding $\eta s$, the intermediate
regions are no longer literally flat potential plateaux; this is the
intended strictly positive background.  The speed ratio $q=\eps/16$
parametrizes the reference step model, not a spatially constant annular
speed for the final perturbed potential.
\end{remark}

\section{The forced infinite-dimensional consequence}\label{sec:road}

Let $H=\ell^2$ with orthonormal basis $(e_k)_{k\ge0}$ and set
\begin{equation}\label{eq:operator}
 \lambda=\sqrt{1+\eps}>1,\qquad Ae_k=\lambda^{2k}e_k.
\end{equation}
For $r,s>0$ define the Gevrey space by
\begin{equation}\label{eq:gevrey}
 \mathcal G_{r,s}(A)=
 \left\{z\in H:\sum_{k\ge0}|\langle z,e_k\rangle|^2
           e^{r\lambda^{k/s}}<\infty\right\}.
\end{equation}
This fixes explicitly the order of the two parameters in our notation.

\begin{corollary}[A forced Kirchhoff blow-up example]\label{cor:blowup}
For a profile $m$ supplied by Theorem~\ref{thm:main} and the operator
\eqref{eq:operator}, there exist $T_\infty<\infty$, a forcing term
\[
 f\in C^0([0,\infty);\mathcal G_{r,s}(A))
           \qquad(s>1,\ r>0),
\]
and a solution of
\begin{equation}\label{eq:forced}
 u''+m(\|A^{1/2}u\|^2)Au=f(t)
\end{equation}
on $[0,T_\infty)$ such that, for every $\alpha>0$,
\begin{equation}\label{eq:blowup}
 \limsup_{t\uparrow T_\infty}
 \bigl(\|A^\alpha u'(t)\|^2+
           \|A^{\alpha+1/2}u(t)\|^2\bigr)=\infty,
\end{equation}
and $u'(t)$ has no limit in $H$ as $t\uparrow T_\infty$.
The solution is of class $C^2$ in every fixed Gevrey space on
$[0,T_\infty)$, including orders $s<1$.
\end{corollary}

\begin{proof}
In the notation of \cite[Definition~2.1]{GGroad}, take
\[
 v=x,\qquad w=y,\qquad \lambda=\sqrt\nu.
\]
The two differential equations agree exactly, because
$\lambda^2=\nu$. The nontriviality condition follows from the positive
energy. The two exponential bounds in that definition are precisely
\eqref{eq:hca-decay}: $v,v'$ decay in the future, and $w,w'$ decay in
the past. Finally $m\in C^1$ and $m\ge\eta>0$. Thus every assumption of
\cite[Theorem~2.3]{GGroad} is satisfied, and the stated corollary follows.
The numbering and the $\limsup$ formulation refer to the cited arXiv
version; our definition \eqref{eq:gevrey} fixes the order of the Gevrey
parameters explicitly.
\end{proof}

The role of the forcing in this implication is worth recalling. The
heteroclinic itself takes infinite time to reach either simple mode. In
the road-map construction, long pieces of its exponentially small tails
are replaced by exact simple modes using cutoffs; the residual of the
equation becomes a small forcing term. Rescaling such finite bridges to
successive frequencies allows their durations to be summable. This is
why the conclusion is a forced infinite-dimensional example, even though
the dynamical input is an unforced planar connection. The quantitative
concatenation and forcing estimates are those of \cite{GGroad} and are
not re-proved here.

\begin{remark}
The construction in \cite{GGroad} starts from a single-mode initial datum.
Its time-regularity refinement \cite[Remark~2.4]{GGroad} also applies to
the smooth coefficient constructed here.  The corollary is a consequence
of that road map; the infinite concatenation is not re-proved in this
paper.  The result does not remove the forcing term and does not assert
blow-up for an arbitrary preassigned operator or bounded spatial domain.
\end{remark}

\begin{remark}[A possible periodic realization]\label{rem:torus}
The operator in Corollary~\ref{cor:blowup} is chosen so that its selected
frequencies form an exact geometric progression.  This exact self-similarity
appears to be a convenience rather than a structural requirement.  Indeed,
once the smooth coefficient $m$ has been fixed, transversality implies that
the planar heteroclinic connection persists for frequency ratios $\lambda$
in a small interval around the value constructed above, with uniform
hyperbolicity after restricting to a compact subinterval.  A corresponding
version of the road-map argument with successive frequency ratios in that
interval would allow one to choose integer frequencies $n_k$ with
$n_{k+1}/n_k$ converging to the reference ratio, and hence to work with the
spectrum $n^2$ of the one-dimensional periodic Laplacian.  We do not pursue
or use this extension here.
\end{remark}

\section{What the argument does and does not use}\label{sec:scope}

The argument separates four logically independent steps.  The
nearly circular asymptotics give the explicit map
\eqref{eq:limitmap}.  The contraction
Lemma~\ref{lem:contraction} gives a bi-infinite orbit, and its linearized
contraction yields transversality.  Proposition~\ref{prop:C1} transfers
that connection to ellipses with small but positive eccentricity.
Finally Lemma~\ref{lem:layer} treats transmission and reflection within
the same smooth fast-time system, allowing the successive perturbations
in $\delta$ and $\eta$.

The exact value $1/\sqrt{63}$ proves contraction without numerical
verification.  Likewise, $d_*<2\pi/3$ provides the strict transmission
margin without evaluating $\rho$ numerically.  Neither $\eps_0$ nor
$\delta_0$ or $\eta_0$ has been numerically certified.  The particular
fixed-frequency example $\nu=2$, $\alpha=1/2$, $q=1/4$ from the earlier
double-hyperbolicity calculation is not the example established here.
The present family instead has $\nu\downarrow1$ and $q=\eps/16$.

The final planar Hamiltonian is smooth and coercive: since
$m\ge\eta>0$, its conserved energy bounds both position and velocity.
The heteroclinic is therefore a global bounded planar trajectory, not a
finite-dimensional blow-up. The blow-up in Corollary~\ref{cor:blowup}
occurs only after applying the forced infinite-dimensional construction.

The proof uses the contraction principle, smooth dependence for ordinary
differential equations and transverse hitting times, and the local stable
manifold theorem with its $C^1$ persistence consequence. The soft-layer
estimate itself is derived here, including both transmissions; it is not
an extension of a reflecting-wall theorem.  No quantitative lower bound
for the admissible values of $\eps$, $\delta$, or $\eta$ is needed for the
existence result.

\subsubsection*{\centering Acknowledgments}

We would like to express our deepest gratitude to Alberto Arosio, who played a decisive role in bringing the Kirchhoff equation to the attention of the Italian mathematical community in the early 1980s. His genuine enthusiasm for this equation was contagious for us already when we were young students. Since then, the Kirchhoff equation has repeatedly returned to the center of our research activity and of our mathematical thoughts. The ideas that eventually led to the counterexample have matured, in different forms, over almost thirty years. In this sense, this paper owes much to the curiosity and enthusiasm that Alberto transmitted to us at the very beginning.

\end{document}